\documentclass[10pt]{amsart}
\usepackage{amsmath,amssymb,amsthm}

\usepackage{float}
\usepackage[all,cmtip]{xy}
\usepackage{tikz}
\usetikzlibrary{matrix}
\usetikzlibrary{cd}
\usepackage{subfig}
\usepackage[colorlinks,hyperindex,linkcolor=blue]{hyperref}
\hypersetup{
 pdfauthor = {Luis Hernández-Corbato and Javier Martínez-Aguinaga},
 pdftitle= {On the Topology of loops of contactomorphisms and Legendrians in non-orderable manifolds},
 pdfsubject = {Differential Geometry, Geometric Topology, h-principle}}

\newcommand{\R}{{\mathbb{R}}}

\newcommand{\NS}{{\mathbb{S}}}

\newtheorem{proposition}{Proposition}
\newtheorem{theorem}[proposition]{Theorem}
\newtheorem{definition}[proposition]{Definition}
\newtheorem{lemma}[proposition]{Lemma}

\newtheorem{corollary}[proposition]{Corollary}
\newtheorem{remark}[proposition]{Remark}

\makeatletter
\newcommand{\superimpose}[2]{%
  {\ooalign{$#1\@firstoftwo#2$\cr\hfil$#1\@secondoftwo#2$\hfil\cr}}}
\makeatother

\title{On the topology of loops of contactomorphisms and Legendrians in non-orderable manifolds}


\author{Luis Hernández-Corbato}
\address{Universidad Complutense de Madrid, Departamento de Álgebra, Geometría y Topología, Facultad de Ciencias
Matemáticas. 28040 Madrid, Spain}
\email{luishcorbato@mat.ucm.es}

\author{Javier Mart\'{i}nez-Aguinaga}
\address{Universidad Complutense de Madrid, Departamento de Álgebra, Geometría y Topología, Facultad de Ciencias
Matemáticas, and Instituto de Ciencias Matemáticas CSIC-UAM-UC3M-UCM, C. Nicolás Cabrera, 13-15.
28049 Madrid, Spain}
\email{frmart02@ucm.es}

\emergencystretch=\maxdimen
\begin{document}

\begin{abstract}

We study the global topology of the space $\mathcal L$ of loops of contactomorphisms of a non-orderable closed contact manifold $(M^{2n+1}, \alpha)$. We filter $\mathcal L$ by a quantitative measure of the ``positivity'' of the loops and describe the topology of $\mathcal L$ in terms of the subspaces of the filtration. In particular, we show that the homotopy groups of $\mathcal L$ are subgroups of the homotopy groups of the subspace of positive loops $\mathcal L^+$.
We obtain analogous results for the space of loops of Legendrian submanifolds in $(M^{2n+1}, \alpha)$.

\end{abstract}
\maketitle
\addtocontents{toc}{\setcounter{tocdepth}{1}}

\section{Introduction}\label{intro}

\subsection{Context and motivation}

The notion of orderability (respectively, non-orderability) was introduced by Eliashberg and Polterovich in \cite{EP}. It is a relevant property in the study of groups of contactomorphisms and contact manifolds with remarkable topological implications and relations to several other problems such as the squeezing of contact domains \cite{EKP}.

The classification of contact manifolds into orderable or non-orderable has been the subject of a significant number of articles \cite{AM, Bh, CP, CPS, EKP, EP, CMM, We}. Likewise, the study of orderability for Legendrian isotopy classes and the existence of positive loops of Legendrians has been extensively studied as well \cite{CCR, CN1, CN2, CFP, Liu, PPP}. 


From the work of Eliashberg and Polterovich \cite{EP}, a cooriented closed contact  manifold $(M, \alpha)$ is orderable if and only if it does not admit a contractible positive loop of contactomorphisms. As a consequence, in every non-orderable closed contact manifold we can always construct a homotopy between an arbitrary loop of contactomorphisms and a positive loop. Nonetheless, as proved by Eliashberg, Kim and Polterovich in \cite{EKP}, 
a contractible positive loop  of contactomorphisms may not be homotopic to the constant loop through positive loops. In the standard tight contact sphere $\left(\NS^{2n-1}, \xi_{std}\right)$, $2n\geq 4$,  there is a contractible positive loop such that every homotopy to the constant loop has an associated Hamiltonian ``sufficiently'' negative somewhere. 

In view of \cite{EKP}, positive loops of contactomorphisms homotopic as loops may well not be homotopic within the space of positive loops. Therefore, in non-orderable manifolds the topology of the space of loops of contactomorphisms $\mathcal L$ and its subspace of positive loops $\mathcal L^+$ can differ. More generally, this indicates that the topology of the space of loops is sensitive to imposing lower bounds on their Hamiltonians; i.e. the topology of the subspaces of loops with Hamiltonians bounded from below by a given constant $c\in\R$ is generally different from that of the entire space of loops $\mathcal L$. 
This note elaborates on the global topology of $\mathcal L$ by establishing connections with the topology of the subspaces of a uniparametric filtration of $\mathcal L$ and, in particular, with the subspace of positive loops $\mathcal L^+$.

\subsection{Statement of the main results}
 We first introduce the power of a loop as a real number that measures the minimum ``positivity'' of its associated Hamiltonian  (Definition \ref{Def:power}). This quantitative measure induces a filtration in $\mathcal L$, the space of loops of contactomorphisms (based at the identity),  by considering, for each fixed $c \in \mathbb R$, the subspace $\mathcal L^c$ of loops of power greater than $c$:
 
 \[
\mathcal{L}^c := \big\lbrace \{\phi_t\} \in \mathcal{L} : pow\left(\{\phi_t\}\right) > c \big\rbrace.
\]

Note that $\mathcal L$ can be seen as the direct limit of $\mathcal L^c$ as $c \to -\infty$. 

For $c > c'$, the natural inclusions $\mathcal L^c \hookrightarrow \mathcal L^{c'} \hookrightarrow \mathcal L$ induce homomorphisms between homotopy groups 

\[i_{c',c} \colon \pi_d(\mathcal L^c) \to \pi_d(\mathcal L^{c'})\quad \text{ and }\quad i_c \colon \pi_d(\mathcal L^c) \to \pi_d(\mathcal L).\]

As mentioned above, in a non--orderable closed contact manifold every loop of contactomorphisms is homotopic to a positive loop or even to a loop of arbitrarily high power. The idea is that we can construct loops of arbitrarily high power by repeatedly concatenating a positive loop, and that the composition of a given loop of contactomorphisms with a powerful loop becomes positive or can even reach a specified power. Since the positive loop is contractible, the previous procedure can be realized through a continuous deformation.

However, to prove that the maps $i_c$ are surjective for every $c \in \mathbb R$ we must control the power of the loops that the basepoint goes through under the homotopy above. Theorem \ref{prop:verypositiveloop} provides the necessary technical result, we construct homotopies between the constant loop and powerful loops in such a way that the power of the involved loops is bounded from below. Without this type of bound, we would not be able,
for example, to guarantee the existence of a homotopy of positive loops between two homotopic loops
of arbitrarily large power.

 On the contrary, the example from \cite{EKP} shows that the maps $i_c$ and $i_{c',c}$ are generally non-injective. In particular, it follows that there are infinitely many values $c, c'\in\mathbb{R}$ for which the maps $\mathcal L^c \hookrightarrow \mathcal L$ (respectively, $\mathcal L^c \hookrightarrow \mathcal{L}^{c'}$) fail to be weak homotopy equivalences. Nonetheless, we can still establish a meaningful relationship at the level of homotopy groups: $i_c$ turns out to be an isomorphism onto $\pi_d(\mathcal L)$ when restricted to the subgroup of $\pi_d(\mathcal L^c)$ composed of classes represented by loops of sufficiently high power. The precise statement reads as follows.

\begin{theorem}\label{thm:intro}
For any $c_0 \in \mathbb R$ there exists $c_1 \ge c_0$ such that the following two statements hold:

\begin{itemize}

  \item $i_{c_0}$ restricts to an isomorphism between $i_{c_0,c_1}(\pi_d(\mathcal L^{c_1}))$ and $\pi_d(\mathcal L)$. In particular, $i_{c_0}$ is surjective and $\pi_d(\mathcal L)$ is isomorphic to a subgroup of $\pi_d(\mathcal L^{c_0})$.

    \item Furthermore, $i_{c_0,c_1}(\pi_d(\mathcal L^{c_1}))$ can be represented by classes of arbitrarily high power. More precisely, for all $c \ge c_1$, $i_{c_0,c}(\pi_d(\mathcal L^c))$ coincides with $i_{c_0,c_1}(\pi_d(\mathcal L^{c_1}))$ for any choice of basepoint of sufficiently large power $c^*$.
    
\end{itemize}
\end{theorem}


We obtain the following corollary which shows that
the global topology of the space of loops $\mathcal L$ can be described in terms of its subspace of positive loops $\mathcal L^+= \mathcal L^0$.

\begin{corollary}\label{CorollaryPositiveIntro}Every homotopy group $\pi_d(\mathcal L)$ of $\mathcal L$ is isomorphic to a subgroup of $\pi_d(\mathcal L^+)$. More precisely:
\[\pi_d(\mathcal L)\cong
i_{0,c}(\pi_d(\mathcal L^{c}))<\pi_d\left(\mathcal L^+\right)\] 

for every sufficiently large $c>0$ and suitable choices of basepoints, where the isomorphism is induced by the natural inclusion $\mathcal{L}^+ \hookrightarrow \mathcal L$.

\end{corollary}

The same arguments can be adapted to the setting of loops of Legendrians submanifolds. For a fixed Legendrian submanifold $\Lambda_0 \subset M$, we consider the space $\mathfrak L$ of loops of Legendrians based at $\Lambda_0$ filtered by their power. Theorem \ref{thm:intro} and Corollary \ref{CorollaryPositiveIntro} hold if we replace loops of contactomorphisms by loops of Legendrians based at $\Lambda_0$. A more detailed description and adapted statements can be found in Section \ref{sec:Leg}.

\textbf{Acknowledgements.} The authors want to express their gratitude
to Francisco Presas, who suggested that a preliminary idea of making loops of Legendrians gain ``positivity'' works in a much more general setting. They also want to thank Álvaro del Pino for key suggestions about controlling the negativity of a homotopy.  Additionally, the second author wants to thank them both for their support and guidance during his PhD thesis.
The first author acknowledges support from PID2021-126124NB-I00 by MICINN (Spain). The second author aknowledges support from PID2022-142024NB-I00 by MICINN (Spain).

\section{Spaces of loops of Contactomorphisms}\label{Sec1}

\subsection{Preliminaries} Throughout this note, we work with a closed cooriented contact manifold $(M, \alpha)$. Let $\{\varphi_t: t \in \NS^1\}$ be a loop of contactomorphisms of $M$. 
All the loops of contactomorphisms considered in this article are based at the identity map, $\varphi_0 = \varphi_1 = \mathrm{id}$, where we think of $\NS^1$ as $[0,1]$ with glued endpoints. Unless otherwise stated, loops are differentiable in the time parameter $t$. The Hamiltonian associated to $\{\varphi_t\}$ is

\[
H_t(x) = H(x,t) = \alpha\left( \frac{d}{dt}\varphi_t(x) \right).
\]

A loop is called \textbf{positive} (resp. non--negative) if the Hamiltonian is positive (resp. greater than or equal to 0). Note that this property is independent of the parametrization of the loop: if $s = s(t)$ is a reparametrization of $[0,1]$ then the Hamiltonian $G_s(x)$ associated to the loop $\{\varphi_s = \varphi_{s(t)}\}$
satisfies

\begin{equation}\label{HamiltonianReparametrization}
G_s(x) = s'(t)\cdot H_t(x).
\end{equation}

\begin{definition}\label{Def:power}
The minimum value attained by the Hamiltonian, or by a family of Hamiltonians, will be called \textbf{power} (of the loop or family of loops; in most instances a homotopy of loops). 
\end{definition}

\begin{remark}
   Evidently, positive loops are the loops with positive power. Equation \eqref{HamiltonianReparametrization} shows the dependence of the power under time reparametrizations.
\end{remark}

\begin{remark}\label{rmk:powerpositive}
    Note that a reparametrization of a positive loop does not change the value of 
    
    \[\int_0^1 \min_x H_t(x)\,dt\]
    
    As a consequence, a positive loop $\{\varphi_t\}$ with Hamiltonian $H_t(x)$ can be reparametrized to a loop of power $\int_0^1 \min_x H_t(x) dt$.
\end{remark}

Given two loops of contactomorphisms $\{\varphi_t\}$ and $\{\phi_t\}$ with Hamiltonians $F_t$ and $G_t$, respectively, the Hamiltonian associated to the loop $\{\varphi_t \circ \phi_t\}$
is 

\begin{equation}\label{composition}
H_t(x) = F_t(x) + \lambda_{\varphi_t}(x) \cdot  G_t(\varphi_t^{-1}(x))
\end{equation}

where $\lambda_{\varphi_t} > 0$ denotes the conformal factor of $\varphi_t$; i.e. $\varphi_t^*\alpha = \lambda_{\varphi_t} \alpha$.

A positive loop $\{\varphi_t\}$ of contactomorphisms can be used to transform an arbitrary loop $\{\psi_t\}$ into a positive one $\{\widetilde{\psi}_t\}$ as follows. Concatenate $\{\varphi_t\}$ with itself a large number of times to obtain a loop with high enough power.
Note that the conformal factor of any element in the concatenation is equal to the conformal factor of some element in $\{\varphi_t\}$. Then, compose this powerful loop with $\{\psi_t\}$ to obtain a loop $\{\widetilde{\psi}_t\}$ which, in view of \eqref{composition}, is positive. Evidently, in this fashion we can produce loops of arbitrarily high power as well.

\begin{remark} It is important to point out that if $\{\varphi_t\}$ is contractible, we do not change the homotopy class of the loop in the previous construction. In other words, if there exists a contractible positive loop, every free homotopy class of loops contains elements of arbitrarily high power.
\end{remark}

Following \cite{EP}, let us state a characterization of non-orderability in the closed case in terms of the power.

\begin{theorem}\label{thm:folclore}
Let $(M, \alpha)$ be a closed contact manifold. The following statements are equivalent:
\begin{itemize}
    \item[(i)] $(M, \alpha)$ is non-orderable.
    \item[(ii)] Every loop of contactomorphisms in $(M, \alpha)$ is freely homotopic to a positive loop.
    \item[(iii)] For every $c\in\mathbb R$, every loop of contactomorphisms in $(M, \alpha)$ is freely homotopic to a loop with power greater than $c$.

\end{itemize}
\end{theorem}

\subsection{Construction of homotopies with bounded negativity}

Recall, as stated earlier, that we will work under the assumption that $(M, \alpha)$ is closed and non-orderable. By \cite{EP}, there exists a contractible positive loop of contactomorphisms $\{\varphi_t\}$. Denote by $\{\varphi_{t,s}: t \in \mathbb S^1, s \in [0,1]\}$ the homotopy connecting the constant loop $\{\varphi_{t,0}\}$ to the positive loop $\{\varphi_{t,1}\}=\lbrace \varphi_t\rbrace$, $\varphi_{t,0} = \varphi_{0,s} = \varphi_{1,s} = \mathrm{id}$. Denote by $F_{t,s}(x)$ the Hamiltonian associated to $\{\varphi_{t,s}\}$ and by $\epsilon < \min_{t,s} F_{t,s}$ a strict lower bound of the power of the homotopy. We choose $\epsilon < 0$. Write $0 < \lambda < \lambda'$ to denote fixed lower and upper bounds of the conformal factors of all the maps $\varphi_{t,s}$. In view of Remark \ref{rmk:powerpositive}, let us reparametrize in $t$ the homotopy $\{\varphi_{t,s}\}$ for values of $s$ close to 1 (and keep the same notation for simplicity) so that

\[
p_0 = pow(\{\varphi_t\}) = \min_{x,t} F_{t,1} = \int_0^1 \min_x F_{t,1}\,dt
\]

and the power of the homotopy remains bounded from below by $\epsilon$. Observe that $p_0 > 0$. We stick hereafter to the notations introduced in this paragraph.

\begin{figure}[h]
	\centering
	\includegraphics[width=1\textwidth]{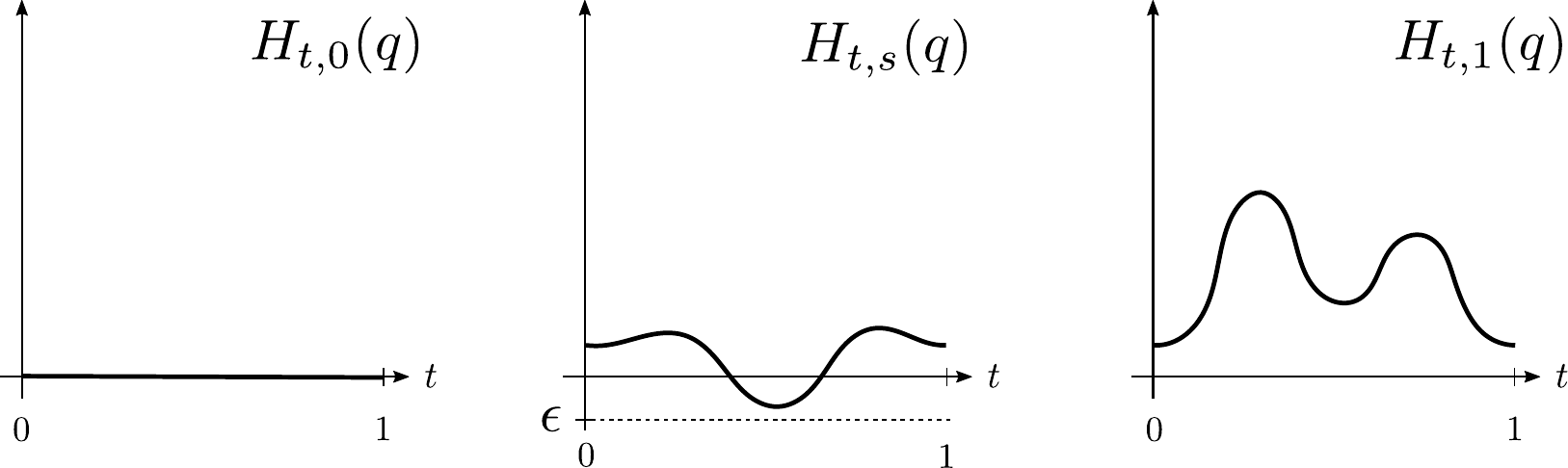}
	\caption{Schematic depiction of the family of Hamiltonians $H_{t,s}(q)$ over a point $q\in M$. Note that the (universal) lower bound $\epsilon$ is depicted in the second frame. This family connects the Hamiltonian associated to the constant identity loop (first frame) with the Hamiltonian of the positive contractible loop $\varphi_t$ (third frame).\label{insertion1}}
\end{figure}

\begin{theorem}\label{prop:verypositiveloop}
For every $C > 0$, there exists $N \in \mathbb N$, a positive loop $\{\varphi^N_t\}$ and a homotopy of loops $\{\varphi^N_{t,s}\}$ from the constant loop $\{\varphi_{t,0}^N\}$ to $\{\varphi^N_t = \varphi^N_{t,1}\}$ such that:
\begin{itemize}
\item[(i)] 
The power of the positive loop is greater than $C$; i.e. $pow\left(\{\varphi^N_t\}\right)\geq C$.
\item[(ii)] The power of the homotopy $\{\varphi^N_{t,s}\}$ is bounded from below by $\epsilon$.
\item[(iii)] $\lambda$ and $\lambda'$ are universal lower and upper bounds for the conformal factors of $\varphi^N_{t,s}$ for all $t,s$.
\end{itemize}
\end{theorem}

\begin{remark}
The notation $\varphi^N_t$ chosen for the positive loops in the statement is justified by the fact that $\varphi^N_t$ will be (as will become apparent in the proof below) a time reparametrization of the $N$--times concatenation of the loop $\{\varphi_t\}$. 
\end{remark}

\begin{proof}
Firstly, the proof constructs inductively homotopies from the constant loop to a (reparametrization of a) multiple concatenation of the positive loop of contactomorphisms $\{\varphi_t\}$ in a way that the power of the homotopies is bounded from below by $2\epsilon$. Later, we will have to modify (reparametrize) the homotopies to fulfill (i). Finally, a clever observation sharpens the power of the homotopy to $\epsilon$.

Let us describe the essence of the inductive step. Suppose we have a non--negative loop of contactomorphisms $\{\psi_t\}$ whose conformal factors are bounded by $\lambda$ and $\lambda'$ as in (iii).
It might be useful to keep in mind that later we will set $\psi_t = \varphi^n_t$. However, let us argue in general. Our goal is to deform $\{\psi_t\}$ into the concatenation of $\{\psi_t\}$ and $\{\varphi_t\}$. The task is divided into two steps.

We warn the reader that the we will have to momentarily consider loops that are not differentiable at some time values $t_i$. This issue will be addressed later in the proof.

\textbf{Step 1}. Consider the family of loops $\{\widetilde{\psi}_{u,s} = \psi_{t(u,s)}: 0 \le u \le 1\}$ where $t(u,s) = \min\{(1+s)u,1\}$ and $s \in [0,1]$. It defines a homotopy between $\{\psi_t\}_{t\in[0,1]}$, for $s = 0$, and $\{\widetilde{\psi}_u = \widetilde{\psi}_{u,1}\}_{u\in[0,1]}$. Note that $\widetilde{\psi}_u$ is equal to the identity for $u \in [1/2,1]$.

    Roughly speaking, this first deformation just ``shrinks'' the support (in time) of the initial loop to the first half on the time-interval (see Figure \ref{shrinking}).  Clearly, this homotopy consists of non-negative loops of contactomorphisms. Also note that it may be non differentiable at two time values (cf. the discontinuities of the Hamiltonian in Figure \ref{shrinking}). This issue will be addressed at the end of this stage. 

\begin{figure}[h]
	\centering
	\includegraphics[width=0.76\textwidth]{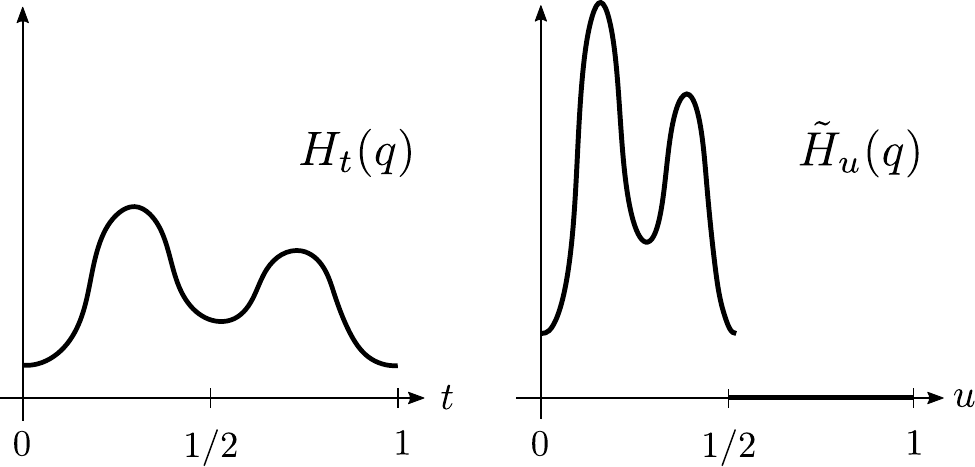}
	\caption{Schematic depiction of Step 1 in the proof: this first part of the homotopy shrinks the time-support of the Hamiltonian to the first half-interval. The left frame depicts the Hamiltonian $H_t$ of $\{\psi_t\}$ over a point $q\in M$ and the right one represents the Hamiltonian $\tilde{H}_u$ of $\{\tilde\psi_u=\psi_{t(u)}\}$ over $q\in M$. \label{shrinking}}
\end{figure}

\textbf{Step 2.}
We will now define the second part of the homotopy (which must then be concatenated to the first part defined in the previous step). 
Roughly speaking, it ``inserts'' the homotopy $\{\varphi_{t,s}\}$ in the interval $1/2 \le u \le 1$ (see Figure \ref{insertion2}). More precisely, define:

\[
\widehat{\psi}_{u,s} = 
\begin{cases}
\widetilde{\psi}_u & \text{if } u \leq \frac{1}{2},\\
\varphi_{2u-1, s} & \text{if } u \geq \frac{1}{2}.
\end{cases}
\]

Note that $\widetilde{\psi}_u$ and $\varphi_{2u-1,0}$ are equal to the identity map if $u \ge 1/2$ so the homotopy $\{\widehat{\psi}_{u,s}: 0 \le s \le 1\}$ starts at the loop $\{\widetilde{\psi}_u\}$.
 
\begin{figure}[h]
	\centering
	\includegraphics[width=1\textwidth]{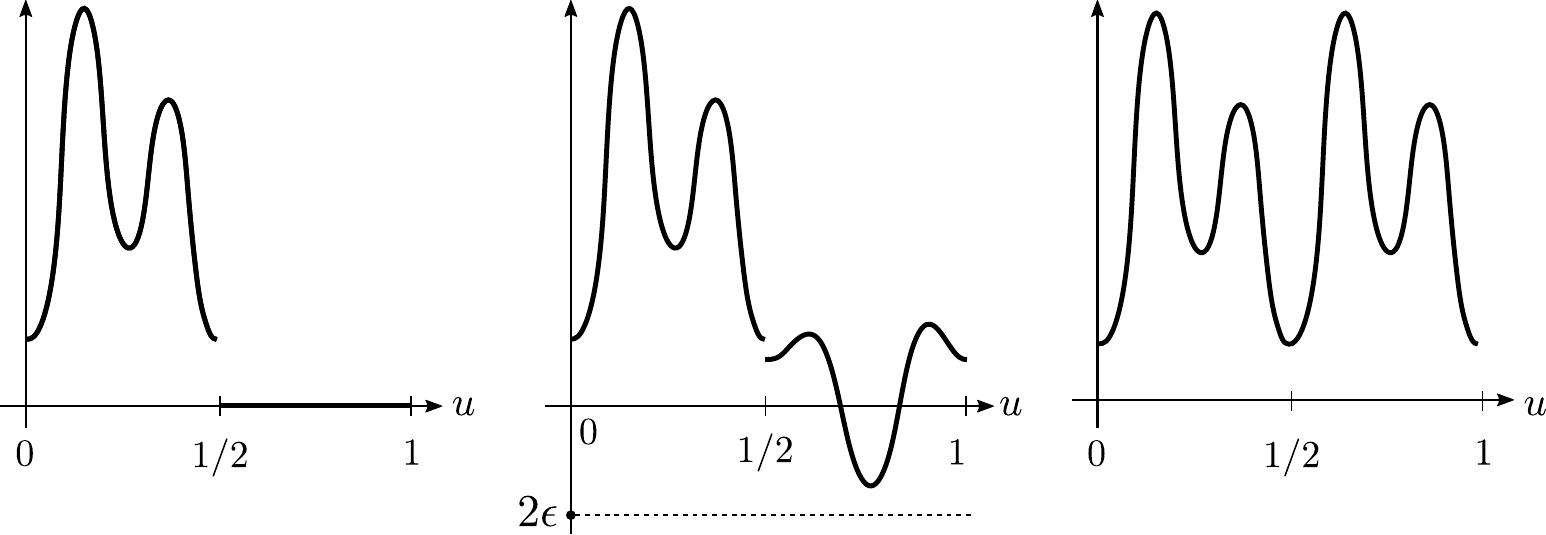}
	\caption{Schematical depiction of Step 2 in the proof for the particular case of $\psi_t=\varphi_t$: ``insertion'' of the homotopy $\{\varphi_{t,s}\}$ in the interval $1/2 \le u \le 1$. Note that the power is bounded from below by $2\epsilon$ due to the reparametrization in time (compare with Figure \ref{insertion1}).\label{insertion2}}
\end{figure}

The homotopy ends at a positive loop $\widehat{\psi}_{u}:=\widehat{\psi}_{u,1}$ that is the concatenation of $\{\psi_t\}$ and $\{\varphi_t\}$. 
In particular, if $H_t$ denotes the hamiltonian associated to $\{\psi_t\}$

\begin{equation}\label{eq:meanHamiltonian1}
\int_0^1 \min \widehat{H}_u \, du = \int_0^{1/2} \min 2 H_{2t}\,dt + \int_{1/2}^1 \min 2F_{2t-1,1}\,dt = \int_0^1 \min H_t\,dt + p_0
\end{equation}

By concatenation of Steps 1 and 2, we obtain a homotopy between $\{\psi_t\}$ and $\{\widehat{\psi}_u\}$ that consists of non--negative loops until we ``insert'' a copy of $\{\varphi_t\}$. The power of the first step is non--negative, while the power of the second step is twice the power of $\{\varphi_{t,s}\}$ due to the time reparametrization. Overall, $2\epsilon$ bounds from below the power of the homotopy (recall that $\epsilon < 0$).

Last, note that, by construction, every contactomorphism in any of the loops of the homotopy is either $\psi_t$ or $\varphi_{t,s}$ for some $t$ and $s$. Thus the conformal factors of the homotopy remain bounded between $\lambda$ and $\lambda'$. 

Now, define homotopies in the following fashion. First, take the homotopy between the constant loop and $\{\varphi_t\}$. Then, starting with the positive loop $\{\varphi_t\}$ as $\{\psi_t$\}, apply inductively the procedure described in Steps 1 and 2 to obtain positive loops $\{\varphi^n_t\}$ for any $n \ge 1$. For each $n \ge 1$, we obtain a homotopy $\{\varphi_{t,s}^n\}$ between the constant loop and $\{\varphi_t^n\}$ that passes through $\{\varphi_t\}$ and $\{\varphi_t^k\}$ for all $k < n$ in this order.
As explained above, the homotopies have power bounded from below by $2\epsilon$ and the conformal factors lie between $\lambda$ and $\lambda'$. Let us analyze the Hamiltonians $F_{t}^n$ associated to $\{\varphi_t^n\}$ in order to reparametrize these positive loops to bound their power. In view of \eqref{eq:meanHamiltonian1}, by induction,

\begin{equation}\label{eq:meanHamiltonian2}
    \int_0^1 \min F_t^{n}\,dt = \int_0^1 \min F_t^{n-1}\,dt + p_0 = n\,p_0.
\end{equation}

\textbf{Discontinuities.} Let us address briefly a technical detail left above in Steps 1 and 2 that carries onto the construction of the homotopies: we may lose differentiability in up to two values of the time parameter $u$, say $u_0, u_1$ (more precisely, $u_0=0$, $u_1$ moves from $1$ to $1/2$ in the first step and $u_1 = 1/2$ for all loops in the second step). See Figure \ref{insertion2}. Note, however, that the one-sided limits of the Hamiltonians at $u_i$ always exist.

One way to recover differentiability is to slow down the loop in a neighborhood of each $u_i$ at the cost of making the Hamiltonian equal to zero. However, we can do better for the final loop $\{\widehat{\psi}_u\}$ since in the inductive step it is equal to $\{\varphi_t^n\}$, a concatenation of $n$ times the original positive loop $\{\varphi_t\}$ travelled at different speeds. A reparametrization that equalizes the parameter speed on both sides of $u_i$ is enough to regain differentiability while keeping positivity.

If the reparametrization is small and supported close to $u_i$ it has negligible effect on the power of the homotopies, that remain greater than $2\epsilon$. Neither the conformal factors nor expression \eqref{eq:meanHamiltonian2} are affected by a reparametrization at all (recall Remark \ref{rmk:powerpositive}).

\textbf{Power of the positive loop.} The final stage of the proof concerns the power. We have to modify further the positive loops and the homotopies so that property (i) in the statement is satisfied. The loops $\{\varphi_t^n\}$ are positive so, by Remark \ref{rmk:powerpositive}, can be reparametrized so that their power is equal to 

\begin{equation}\label{ExplicitBound}
\int_0^1 \min F_t^n\,dt = n\,p_0.
\end{equation}

Therefore, we can reparametrize in $t$ the homotopy $\{\varphi_{t,s}^n: 0 \le s \le 1\}$, only for values of $s$ sufficiently close to 1, so that the power of the final loop is equal to $n\,p_0$ while the power of the homotopy itself remains bounded by $2\epsilon$. For the sake of simplicity, we keep the same notations $\{\varphi_t^n\} / \{\varphi_{t,s}^n\}$ for the reparametrized loop/homotopy.

\textbf{Power of the homotopy.} It is clear from the proof that the factor 2 in the bound $2\epsilon$ is inverse to the length of the subinterval, $[1/2,1]$, where Step 2 takes place, that is, where the homotopy $\{\varphi_{t,s}\}$ is ``inserted'' after shrinking $\{\psi_t\}$. Should the subinterval be $[\delta, 1]$, the lower bound on the power of the homotopies $\{\varphi_{t,s}^n\}$ would be $\epsilon/(1-\delta)$ for all $n$. Therefore, if we choose $\epsilon', \delta$ so that 

\[\epsilon \le \epsilon'/(1-\delta) < \epsilon' < pow(\{\varphi_{t,s}\})\]

and start the proof from scratch with $\epsilon'$ instead of $\epsilon$ and take $\delta \in [0,1]$ as pivot point instead of $1/2$, we produce homotopies whose power is bounded from below by $\epsilon'/(1-\delta) \ge \epsilon$.\end{proof}

\begin{remark}
Note that, although we will not make further use of it, we explicitly computed the power of $\{\varphi^N_t\}$ within the proof: $pow(\{\varphi_t^N\}) = Np_0$.
\end{remark}

\subsection{Homotopy groups of the space of loops of contactomorphisms}

Denote by $F_{t,s}^N \colon M \to \mathbb R$ the Hamiltonians associated to the loops $\{\varphi_{t,s}^N\}$ of the homotopies constructed in Theorem \ref{MainThm}. Recall that, for every $C$ there exists $N\ge 1$ such that $F^N_{t,1} \ge C$ and $F^N_{t,s} > \epsilon$.

Suppose now that $\{\phi_t\}$ is a loop of contactomorphisms with associated Hamiltonian $G_t \colon M \to \mathbb R$. By \eqref{composition}, the Hamiltonian associated to $\{\varphi^N_{t,s} \circ \phi_t\}$ is: 

\begin{equation}\label{eq:composition}
F^N_{t,s} + \lambda^N_{t,s}\cdot G_t \circ (\varphi^N_{t,s})^{-1}
\end{equation}

where $\lambda^N_{t,s} \colon M \to \mathbb R^+$ denotes the conformal factor of $\varphi^N_{t,s}$, which takes values in $[\lambda, \lambda']$ by Theorem \ref{prop:verypositiveloop}.

By compactness of $M$, the values of $G_t$ are bounded and so are the conformal factors $\lambda^N_{t,s}$, so (\ref{eq:composition}) depends mainly on $F_{t,s}^N$ or, equivalently, on $N$ (and $C$). More precisely, if $p = pow(\{\phi_t\})$,

\begin{equation}\label{eq:cotapowerhomotopia}
pow(\{\varphi^N_{t,s} \circ \phi_t\}) \ge \epsilon + \min\{\lambda \cdot p, \lambda' \cdot p\}
\end{equation}

Let us emphasize that the lower bound is independent of $N$. For $s = 1$, $\epsilon$ can be replaced by $C$ in the formula and we conclude that $\{\varphi^N_{t} \circ \phi_t\}$ is a positive loop for $C > \max\{0, -\lambda'p\}$.


The previous estimates are trivially valid for a continuous family $\{\{\phi_{t}^{\theta}\}: \theta \in \Theta\}$ of loops of contactomorphisms indexed by a compact set $\Theta$ provided $p\in\mathbb{R}$ is a uniform lower bound of their power. We apply them to the case $\Theta = \NS^d$ to work with homotopy groups of the space of loops of contactomorphisms. 

Denote by $\mathcal{L}$ the space of loops of contactomorphisms in $(M, \alpha)$ and by $\mathcal{L}^+$ the subspace of positive loops. For a fixed $c \in \mathbb R$, denote by $\mathcal L^c$ the subspace of loops of power greater than $c$;

\[
\mathcal{L}^c := \big\lbrace \{\phi_t\} \in \mathcal{L} : pow\left(\{\phi_t\}\right) > c \big\rbrace.
\]

Clearly, $\mathcal L$ can be seen as the direct limit of $\mathcal L^c$ as $c \to -\infty$. The purpose of the rest of the section is to relate the topology of $\mathcal L$ with the topology of the subspaces $\mathcal L^c$.

For $c > c'$, the natural inclusions $\mathcal L^c \hookrightarrow \mathcal L^{c'} \hookrightarrow \mathcal L$ define homomorphisms 

\[i_{c',c} \colon \pi_d(\mathcal L^c) \to \pi_d(\mathcal L^{c'})\quad \text{ and }\quad i_c \colon \pi_d(\mathcal L^c) \to \pi_d(\mathcal L).\]

Basepoints have been omitted from the notation. By arguments that are explained in some detail later, in the proof of Theorem \ref{MainThm}, basepoints are loops whose power is large enough. Remember that, by Theorem \ref{thm:folclore}, every path connected component of $\mathcal L$ contains loops of arbitrarily high power.
Power estimates above imply that any sphere in $\mathcal L$ can be deformed to a sphere in $\mathcal L^c$ and the path traced by a basepoint remains in $\mathcal L^c$ provided its power is large. This is key to prove that the homomorphisms $i_c$ are surjective for every $c \in \mathbb R$ in Theorem \ref{MainThm}.

On the contrary, the example in \cite{EKP} of a contractible positive loop which cannot be deformed to the constant loop through positive loops shows that $i_c$ is not injective in general. In particular, one cannot expect the existence of weak homotopy equivalences between $\mathcal L$ and some $\mathcal L^c$, for instance the subspace of positive loops.
However, for every $d \ge 0$ and $c_0 \in \mathbb R$, $\pi_d(\mathcal L)$ is isomorphic to a relevant subgroup of $\pi_d(\mathcal L^{c_0})$, namely, the subgroup where the image of the maps $i_{c_0, c}$ stabilize as $c$ grows (see Theorem \ref{MainThm} below).

Before entering into the result and its proof let us state an elementary lemma in Algebraic Topology that allows to deform a free homotopy of (pointed) spheres into a pointed homotopy; that is, a homotopy that fixes the basepoint. This standard deformation is easily defined using the path traced by the basepoint. Note that the change of basepoint isomorphism for homotopy groups employs similar constructions.


\begin{lemma}\label{lem:basepoint}
Let $X$ be a topological space and $\lbrace f_s\rbrace_{s\in[0,1]}:\NS^d\to X$ a free homotopy of spheres. Then $\lbrace f_s\rbrace_{s\in[0,1]}$ is homotopic to a pointed homotopy $\lbrace g_s\rbrace_{s\in[0,1]}:(\NS^d,*)\to (X, q_0)$, where $q_0 = f_0(*)$, satisfying:

\begin{itemize}
    \item[i)]  $g_s(*) = q_0$ for all $s\in[0,1]$, 
    \item[ii)] $Im(g_s)=Im(f_s)\cup\{f_u(*): 0\leq u\leq s\}$ for all $s\in[0,1]$.
\end{itemize}
\end{lemma}

\begin{remark} In order to have a picture in mind, let us say that we can regard $g_s$ as the ``balloon'' $f_s$ tied to $q_0$ with the ``string'' $\{f_u(*): 0 \le u \le s\}$.

\end{remark}

\begin{proof}
Suppose, in order to ease the forthcoming definition, that each $f_s$ is a map $[-1,1]^d \to X$ such that the boundary of $[-1,1]^d$ (identified to $*$) is mapped into a single point, say $q_s$.

    For every $s \in [0,1]$ define $g_s \colon [-1,1]^d \to X$ as
    
\begin{equation}
g_s(z) = \begin{cases}
f_s((1+s)z) & \text{if} \,\, m = \max\{|z_1|, \ldots, |z_d|\} \le 1/(1+s) \\
q_{(1+s)(1-m)} & \text{otherwise}\\
\end{cases}
\qquad \text{where}\, z = (z_1, \ldots, z_d).
\end{equation}

Since $g_s(\partial [-1,1]^d) = q_0$, we can see $\{g_s\}_{s \in [0,1]}$ as a homotopy of pointed maps $g_s \colon (\NS^d,*) \to (X,q_0)$ such that $g_0 = f_0$. Consequently, $\lbrace g_s\rbrace_{s\in[0,1]}$ is (freely) homotopic to $\lbrace f_s\rbrace_{s\in[0,1]}$. Property (ii) and the picture of $g_s$ described in the previous remark follow from the definition.\end{proof}


\begin{remark}\label{rmk:basepointBall}
The previous Lemma can be readily adapted to transform a free homotopy of closed balls in $X$ (continuous images of the unit $d$--dimensional disk) into a pointed homotopy of (pointed) closed balls. Indeed, given such a free homotopy, we can apply Lemma \ref{lem:basepoint} to the homotopy of the boundary sphere and just extend the homotopy of capping balls in the obvious manner.
\end{remark}

We now state and prove the main result of the paper, which is presented as Theorem \ref{thm:intro} in the Introduction.

\begin{theorem}\label{MainThm}
For any $c_0 \in \mathbb R$ there exists $c_1 \ge c_0$ such that the following two statements hold:

\begin{itemize}

  \item $i_{c_0}$ restricts to an isomorphism between $i_{c_0,c_1}(\pi_d(\mathcal L^{c_1}))$ and $\pi_d(\mathcal L)$. In particular, $i_{c_0}$ is surjective and $\pi_d(\mathcal L)$ is isomorphic to a subgroup of $\pi_d(\mathcal L^{c_0})$.

    \item Furthermore, $i_{c_0,c_1}(\pi_d(\mathcal L^{c_1}))$ can be represented by classes of arbitrarily high power. More precisely, for all $c \ge c_1$, $i_{c_0,c}(\pi_d(\mathcal L^c))$ coincides with $i_{c_0,c_1}(\pi_d(\mathcal L^{c_1}))$ for any choice of basepoint of sufficiently large power $c^*$.
    
\end{itemize}
\end{theorem}

\begin{figure}[h]
	\centering
	\includegraphics[width=0.6\textwidth]{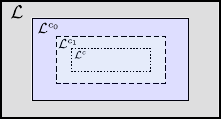}
	\caption{Flag of inclusions $\mathcal L\supset \mathcal L^{c_0}\supset \mathcal L^{c_1}\supset \mathcal L^{c}$. By Theorem \ref{MainThm}, for fixed $c_0\in\mathbb{R}$ there exists $c_1\geq c_0$ so that all subgroups $i_{c,c_0}(\pi_d(\mathcal L^c))$ are equal for $c\geq c_1$. In other words, there is a stable range for $c\in\mathbb{R}$ (stripped light-blue region) where all the subspaces $\mathcal L^c$ become weakly homotopy equivalent within $\mathcal L^{c_0}$.}
\end{figure}

\begin{proof} Let us start with the second item. We need to check that, for large enough $c_1$ and $c\geq c_1$, every homotopy class in $i_{c_0,c_1}(\pi_d(\mathcal L^{c_1}))$ can be realised by a homotopy class in  $i_{c_0,c}(\pi_d(\mathcal L^c))$. Equivalently, we need to show that every sphere $\{\{\phi_{t}^\theta\}: \theta \in S^d\}$ in $\mathcal L^{c_1}$ can be homotoped, within $\mathcal L^{c_0}$, to a sphere in $\mathcal L^c$. The choice of basepoint $\{\phi_t^*\}$ for the homotopy groups involved will be any loop of large enough power $c^*$, to be specified later.

In view of \eqref{eq:cotapowerhomotopia}, it suffices to guarantee that $\epsilon + \lambda c_1$ and $\epsilon + \lambda' c_1$ are greater than $c_0$. Indeed, assume that $c_1\geq c_0$ satisfies those estimates and take $c\geq c_1$. By the discussion below \eqref{eq:composition}, we can thus homotope $\{\phi_{t}^\theta\}_{\theta \in \NS^d}$ to $\{\varphi^N_t\circ\phi_{t}^\theta\}_{\theta \in \NS^d}$ within $\mathcal L^{c_0}$ for any $N \ge 1$. For a large enough choice of $N$ (which only depends on $c_1$ and $c$; and thus
valid for any sphere $\{\phi_t^{\theta}\}_{\theta \in \NS^d}$ in $\mathcal L^{c_1}$) we have that $pow(\{\varphi^N_t\circ\phi_{t}^{\theta}\}) > c$ for all $\theta \in \NS^d$, that is, the sphere of loops $\{\varphi^N_t\circ\phi_{t}^{\theta}\}$ lies entirely within $\mathcal L^c$.

The deformation we just described is a free homotopy that moves the basepoint through loops with power bounded from below by $\min\{\epsilon+\lambda c^*, \epsilon+\lambda'c^*\}$ and, thus, if $c^*$ is large enough, the path travelled by such basepoint lies entirely in $\mathcal L^c$. By Lemma \ref{lem:basepoint}, there is a pointed homotopy between $\{\phi_t^{\theta}\}_{\theta \in \NS^d}$ and a sphere in $\mathcal L^c$ based at $\{\phi_t^*\}$ within $\mathcal L^{c_0}$. This yields the claim.

The first item can be proved along the same lines. To check injectivity suppose that $\{\phi_t^{\theta}\}_{\theta \in \NS^d}$ is a sphere that lies in $\mathcal L^{c_1}$ (for the same choice of $c_1$ as above) and is capped by a ball $\{\phi_t^{\overline{\theta}}\}_{\theta \in B^{d+1}}$ in $\mathcal L$. By compactness, if $N$ is large enough, $\{\varphi_{t}^N \circ \phi_t^{\overline{\theta}}\}$ is a ball of loops in $\mathcal L^{c_1}$ (freely) homotopic to $\{\phi_t^{\overline{\theta}}\}$, where the homotopy is given by $\{\varphi_{t,s}^N \circ \phi_t^{\overline{\theta}}\}_{s\in[0,1]}$. By the choice of $c_1$, the path travelled by the basepoint of the sphere during the homotopy is totally contained in $\mathcal L^{c_0}$ and, thus, can be used to define a pointed homotopy of closed balls between $\{\phi_t^{\bar{\theta}}\}_{\theta \in B^{d+1}}$ and a closed ball lying entirely within $\mathcal L^{c_0}$ as mentioned in Remark \ref{rmk:basepointBall}. This yields injectivity. Note that there are no restrictions to the choice of basepoint in this argument other than that it belongs to $\mathcal L^{c_1}$. 

As for surjectivity, note that the argument in the previous paragraph almost does the job. If we start with a sphere $\{\phi_t^{\theta}\}_{\theta \in \NS^d}$ in $\mathcal L$ based at some $\{\phi_t^*\}\in\mathcal L^{c_1}$, we can just consider its deformation into the sphere $\{\varphi_{t}^N \circ \phi_t^{\theta}\}_{\theta \in \NS^d}$ (which for large enough $N$ lies in $\mathcal L^{c_0}$). Once again, we can use the path travelled by its basepoint (which lies in $\mathcal L^{c_0}$ since we took $\{\phi_t^*\}$ in $\mathcal L^{c_1}$) to make the deformation relative to basepoint. \end{proof}

Note that by taking $c_0=0$ in Theorem \ref{MainThm} we obtain the following corollary.

\begin{corollary}\label{CorollaryPositive}
Every homotopy group $\pi_d(\mathcal L)$ of the space of loops of contactomorphisms $\mathcal L$ is isomorphic to a subgroup of the homotopy group of the space of positive loops $\mathcal L^+$. More precisely: 

\[\pi_d(\mathcal L)\cong
i_{0,c}(\pi_d(\mathcal L^{c}))<\pi_d\left(\mathcal L^+\right)\]

for every sufficiently large $c>0$. Moreover, the isomorphism is induced by the inclusion $\mathcal{L}^+ \hookrightarrow \mathcal L$.
\end{corollary}

\begin{proof}
    In view of the proof of Theorem \ref{MainThm}, it is enough to take $c \ge -\epsilon/\lambda$.
\end{proof}

\section{Applications to spaces of loops of Legendrian submanifolds}\label{sec:Leg}

All the arguments above readily translate to the context of loops of Legendrian submanifolds in a rather straightforward fashion. Let us, nonetheless, briefly elaborate on this for the sake of completeness.

Let $\{\Lambda_t\}$, $t \in \NS^1$, be a loop of Legendrians based at a given Legendrian submanifold $\Lambda_0\subset M$. Consider a family of embeddings $j_t:\Lambda\to M$, $j_t(\Lambda) = \Lambda_t$, $0 \le t \le 1$, generating the loop of Legendrians. The associated \textbf{contact Hamiltonian} $h_t$ is given by 

\[h_t(x):=\alpha_{j_t(x)}\left(\frac{dj_t}{dt}(x)\right).\] 


The definition of power naturally adapts for loops of Legendrians.

\begin{definition}
We will call \textbf{power of a loop of Legendrians} (or of a family of loops) to the minimum value $\min_{t,x} h_t(x)$ attained by its hamiltonian $h_t$ (or family of Hamiltonians).
\end{definition}

By the Legendrian isotopy extension theorem, any loop of Legendrians $\{\Lambda_t\}$ can be extended to a contact isotopy $\{\phi_t\}$ starting at the identity. Note, however, that $\{\phi_t\}$ may not be a loop of contactomorphisms because $\phi_1$ may differ from the identity.

A close inspection of the proof of the Legendrian isotopy extension theorem (for example in \cite[Thm. 2.6.2]{GeigesBook}) immediately yields that it can be used parametrically. Given a continuous family of loops of Legendrians $\{\Lambda_t^{\theta}\}_{\theta \in \Theta}$ we obtain a family of contact Hamiltonians $\{h_t^{\theta}\}_{\theta \in \Theta}$ continuous in $\theta$ that generates a family of contact isotopies $\{\phi_t^{\theta}\}_{\theta \in \Theta}$, continuous in $\theta$ as well, that extends $\{\Lambda_t^{\theta}\}_{\theta \in \Theta}$.



Note that the positive contractible loop of contactomorphisms $\{\varphi_t\}$ gives raise to positive contractible loops of Legendrians just by the natural action. If $\{j_t\}$ is a loop of Legendrians, $j_t \colon \Lambda \to M$, with associated Hamiltonian $G_t \colon \Lambda \to \mathbb R$, then the Hamiltonian associated to the composition $\{\varphi^n_{t,s} \circ j_t\}$ is:

\begin{equation}\label{eq:compositionLeg}
F^n_{t,s} + \lambda^n_{t,s} \cdot G_t \circ (\varphi^n_{t,s})^{-1}.
\end{equation}

Thus, the same estimates from Subsection \ref{Sec1} follow for loops and families of loops of Legendrians. In consequence, we recover, in the context of loops of Legendrian submanifolds, results similar to those stated in the previous subsection with analogous proofs.

Fix a Legendrian submanifold $\Lambda_0 \subset M$. Henceforth assume that all the loops of Legendrians are based at $\Lambda_0$.
Denote the space of loops of Legendrians by $\mathfrak{L}$ and by $\mathfrak{L}^+$ its subspace of positive loops. For a fixed $c \in \mathbb R$, denote by $\mathfrak{L}^c$ the subset of loops of power greater than $c$. For $c > c'$, we have the same natural inclusions as earlier:

\[\mathfrak{L}^c \hookrightarrow \mathfrak{L}^{c'} \hookrightarrow \mathfrak{L}\]

which induce homomorphisms at the level of homotopy groups:

\[\iota_{c',c} \colon \pi_d(\mathfrak{L}^c) \to \pi_d(\mathfrak{L}^{c'})\quad \text{ and }\iota_c \colon \pi_d(\mathfrak{L}^c) \to \pi_d(\mathfrak{L}).\]

The discussion on basepoints from the previous section is valid in this context. Basepoints are loops of Legendrians of sufficiently large power.
Clearly, the argument described before Theorem \ref{thm:folclore} which deforms a given loop to a loop above any prescribed power can be translated verbatim to the context of Legendrian submanifolds. 
Likewise, Theorem \ref{prop:verypositiveloop} is the key to maintain under control the power of the loops of the deformation and obtain sharper results. The analogue of Theorem \ref{MainThm} for Legendrian loops is now stated.

\begin{theorem}\label{MainThmLeg}
For every $c_0 \in \mathbb R$ there exists $c_1 \ge c_0$ such that, for all $c \ge c_1$, $\iota_{c_0,c} (\pi_d(\mathfrak{L}^c))$ coincides with $\iota_{c_0,c_1}(\pi_d(\mathfrak{L}^{c_1}))$ for any basepoint of sufficiently large power $c^*$. Furthermore, $\iota_{c_0}$ restricts to an isomorphism between $\iota_{c_0,c_1}(\pi_d(\mathfrak{L}^{c_1}))$ and $\pi_d(\mathfrak{L})$. In particular, $\iota_{c_0}$ is surjective.
\end{theorem}

The same arguments from Theorem \ref{MainThm} prove Theorem \ref{MainThmLeg}. In particular, basepoints in the first statement need to have power $c^*$ such that $c \le \epsilon + \min\{\lambda c^*, \lambda' c^*\}$ whereas any basepoint in $\mathfrak L^{c_1}$ works for the second statement.
Setting $c_0=0$, we get the following corollary, which is analogous to Corollary \ref{CorollaryPositive}.

\begin{corollary}\label{CorollaryPositive2}
Every homotopy group $\pi_d(\mathfrak{L})$ of the space of loops of Legendrians $\mathfrak L$ is isomorphic to a subgroup of the homotopy group of the space of positive loops of Legendrians $\mathfrak{L}^+$. More precisely: 

\[\pi_d(\mathfrak{L}) \cong
\iota_{0,c}(\pi_d(\mathfrak{L}^c))<\pi_d\left(\mathfrak{L}^+\right)\]

for every sufficiently large $c>0$. Such isomorphism is induced by the inclusion $\mathfrak{L}^+ \hookrightarrow\mathfrak{L}$.
\end{corollary}


\begin{thebibliography}{xxxx}



\bibitem{AM} P Albers, WJ Merry. \emph{Orderability, contact non-squeezing, and Rabinowitz Floer homology}.  Journal of Symplectic Geometry 16 (6), 1481--1547, 2018.



\bibitem{Bh} M. Bhupal. \emph{A partial order on the group of contactomorphisms of $\mathbb R^{2n+1}$ via generating functions}. Turkish J. Math. 25
(2001), 125--135.



\bibitem{CCR} B. Chantraine, V. Colin, G.  Dimitroglou Rizell. \emph{Positive Legendrian Isotopies And Floer Theory}. Annales de l'Institut Fourier, 69(4), 1679-1737.

\bibitem{CP} R. Casals, F. Presas. \emph{On the strong orderability of overtwisted 3-folds}. Comment. Math. Helv. 91 (2016), no. 2, 305–316.

\bibitem{CPS} R. Casals, F. Presas, S. Sandon. \emph{On the non-existence of small positive loops of
contactomorphisms on overtwisted contact manifolds}. Journal of Symplectic Geometry, 14
(2016), no. 4, 1013-–1031.

\bibitem{CN1} V. Chernov, S. Nemirovski. \emph{Non-negative Legendrian isotopy in $S(T^*M)$}.
Geom. Topol. 14 (2010), 611-626.

\bibitem{CN2} V. Chernov, S. Nemirovski. \emph{ Universal orderability of Legendrian isotopy
classes}. J. Symplectic Geom. 14 (2016), no. 1, 149–170.

\bibitem{CFP} V. Colin, E. Ferrand, P. Pushkar. \emph{Positive isotopies of Legendrian submanifolds}.
Int. Math. Res. Not. IMRN (2017), no. 20, 6231–6254.

\bibitem{EKP} Y. Eliashberg, S. S. Kim, and L. Polterovich. \emph{Geometry of contact transformations and domains: orderability versus squeezing.} Geometry \& Topology, vol. 10, n. 3, 1635–1747. 2006.

\bibitem{EP} Y. Eliashberg, L. Polterovich. \emph{Partially ordered
groups and geometry of contact transformations}. GAFA 10
(2000), no. 6, 1448–1476.





\bibitem{GeigesBook} H. Geiges. \emph{An Introduction to Contact Topology.} Cambr. Studies in Adv. Math. 109. Cambr. Univ. Press 2008.


\bibitem{CMM} L. Hernández-Corbato, L. Martín-Merchán, F. Presas. \emph{Tight neighborhoods of contact submanifolds}. J. Symplectic Geom. 18 (2020) 1629-1646.



\bibitem{Liu} G. Liu. \emph{Positive loops of loose Legendrian embeddings and applications}.  Journal of Symplectic Geometry 18 (3), 867--887, 2020.


\bibitem{PPP} D. Pancholi, J. L. Pérez, F. Presas. \emph{A simple construction of positive loops of Legendrians}. Arkiv för Matematik, 56, Number 2 (2018), 377 -- 394.



\bibitem{We} P. Weigel. \emph{Orderable Contact Structures on Liouville-fillable Contact Manifolds}. Journal of Symplectic Geometry 13(2),
2013.


\end{thebibliography}
\end{document}